\documentclass[smallextended]{svjour3}
\usepackage{amssymb,amsmath,hyperref,verbatim,cite,graphicx}
\makeatletter \let\cl@chapter\relax \makeatother
\usepackage{cleveref, xurl}
\usepackage{algorithmic}
\usepackage{multirow, longtable}
\usepackage{booktabs}
\usepackage{diagbox}
\usepackage{subfigure}
\usepackage{bbm}
\usepackage{pdflscape}
\usepackage{mathtools}

\usepackage{longtable}
\usepackage{enumerate}
\usepackage[dvipsnames]{xcolor}
\usepackage{tikz}
\usepackage{pgfplots}
\pgfplotsset{width=10cm,compat=1.9}
\usepackage[ruled,vlined,linesnumbered]{algorithm2e}
\usepackage{hyperref}
\usepackage{subcaption}
\hypersetup{
    colorlinks=true,
    linkcolor=blue,
    filecolor=magenta,      
    urlcolor=cyan,
}
\newtheorem{assumption}{Assumption}

\begin{document}
\newcommand{\red}{\color{red}}
\newcommand{\blue}{\color{blue}}
\newcommand{\green}{\color{green}}
\newcommand{\conv}{conv}
\newcommand{\rank}{rank}
\allowdisplaybreaks

\title{Tensor Completion via Integer Optimization}


\author{Xin Chen\textsuperscript{1} \and Sukanya Kudva\textsuperscript{1} \and Yongzheng Dai \and Anil Aswani \and Chen Chen}

\footnotetext[1]{Equal contribution}
%\footnotetext[2]{Equal contribution}
\institute{Xin Chen \at 
            \email{xin.chen2024@berkeley.edu}\\
            IEOR, University of California Berkeley, Berkeley, CA, USA
            \and
            Sukanya Kudva \at 
            \email{sukanya\_kudva@berkeley.edu}\\
            IEOR, University of California Berkeley, Berkeley, CA, USA
            \and
            Yongzheng Dai \at
			\email{dai.651@osu.edu}\\
			ISE, The Ohio State University, Columbus, OH, USA
			\and
            Anil Aswani \at 
            \email{aaswani@berkeley.edu}\\
            IEOR, University of California Berkeley, Berkeley, CA, USA
            \and
			Chen Chen \at
            \email{chen.8018@osu.edu}\\
            ISE, The Ohio State University, Columbus, OH, USA
}

\date{Received: date / Accepted: date}

\maketitle

\begin{abstract}

The main challenge with the tensor completion problem is a fundamental tension between computational effort and the information-theoretic sample complexity rate. Past approaches either achieve the information-theoretic rate but lack practical algorithms to compute the corresponding solution, or have polynomial-time algorithms that require an exponentially-larger number of samples for low estimation error. This paper develops a novel tensor completion algorithm that resolves this tension by achieving both provable convergence (in numerical tolerance) in a linear number of oracle steps and the information-theoretic rate. Our approach formulates tensor completion as a convex optimization problem constrained using a gauge-based tensor norm, which is defined in a way that allows the use of integer linear optimization to solve linear separation problems over the unit-ball in this new norm. Adaptations based on this insight are incorporated into a Frank-Wolfe variant to build our algorithm. We show our algorithm scales-well  using numerical experiments on tensors with up to ten million entries.
\end{abstract}

\section{Introduction}
%\label{submission}
A tensor is a multilinear operator, and it can be represented as an array of numbers referenced by multiple indices. For a tensor $\psi \in \mathbb{R}^{r_1 \times \cdots \times r_p}$ of order $p$, we use the notation $\psi_x:= \psi_{x_1,\ldots,x_p}$ to refer to an entry corresponding to the indices $x = (x_1,\ldots,x_p)$, where $x_i \in [r_i]$ and $[s]:= \{1,\ldots,s\}$. Furthermore, we define some parameters that depend upon the tensor dimensions and order: $\rho = \sum_i r_i$, $\pi = \prod_{i} r_{i}$, and $\mathcal{R} = [r_1]\times\cdots\times[r_p]$. Though vectors and matrices are special cases of tensors, tensors with three or more indices pose unique challenges. Many tensor problems are \textsc{np}-hard \cite{hillar2013}, such as computing: rank, singular values, and nuclear norm. 

In this paper, we address one such difficult problem called \emph{tensor completion}. Here, a small subset of tensor entries are observed -- possibly with noise. Under an assumption of low-rankness, the problem is to fill-in the remaining, unobserved entries -- and remove noise, if any. Since modern datasets are often multidimensional, there are many applications of tensor completion \cite{song2019tensor}, including: recommendation systems \cite{Ge_Caverlee_Lu_2016, Karatzoglou_Amatriain_Baltrunas_Oliver_2010}, information diffusion \cite{Zafarani_Abbasi_Liu_2014}, regression \cite{aswani2016}, computer vision \cite{Duarte_Baraniuk_2012, Plas_De_Moor_Suykens_2011, Tan_Cheng_Feng_Feng_Wang_Zhang_2013}, and bioinformatics \cite{Bazerque_Mateos_Giannakis_2013, Acar_Dunlavy_Kolda_Mørup_2011}. 

\subsection{Related Work}
 The information-theoretic rate for estimation error is $\sqrt{k\cdot\sum_{i} r_{i}/n}$ for a tensor completion problem, where: $k$ is tensor rank, $r_i$ is the $i$-th dimension of the tensor, and $n$ is the number of samples \cite{gandy2011tensor}. Previous works have tried to characterize the tradeoff between complexity and the information-theoretic rate. For example, the approach in \cite{barak2016noisy} uses the Rademacher complexity of a sum-of-squares hierarchy to suggest a suitable norm for tensor completion, and observe a gap between what can be achieved information theoretically and what is attained by their computationally efficient method.

Initial work on tensor completion used decomposition methods based on \textsc{CP} and Tucker decomposition \cite{Tomasi_Bro_2005}. Other approaches accounted for robustness to outliers and data corruptions \cite{Javed_Bouwmans_Jung_2015, Jain_Gutierrez_Haupt_2017}, or imposed fixed-rank constraints \cite{kressner2014low}.  Alternative approaches used various tensor norms as a convex surrogate for tensor rank \cite{montanari2018spectral, yuan2016tensor,barak2016noisy}. Our approach falls into this last category, but we define a tensor norm using a gauge function as a convex surrogate. 
 
Past approaches either traded information-theoretic rate for a computationally efficient algorithm \cite{montanari2018spectral,barak2016noisy}; likewise, heuristics have been developed to compute non-certifiably-optimal solutions to such \textsc{np}-hard problems whose optimal solutions (if found) in principal achieve the information-theoretic rate \cite{kressner2014low,yuan2016tensor, Jain_Gutierrez_Haupt_2017}. In contrast, our algorithm is globally convergent and attains the information-theoretic rate (hence data efficiency), while certifying optimal solutions on instances on general tensors of sizes up to $10^{\times 7}$ within minutes. 

Certain algorithms for special cases of tensor completion have achieved the information-theoretic rate through practical computation. Aswani \cite{aswani2016} studied tensor completion for rank-1 non-negative tensors, formulating it as conic optimization with exponential constraints. Symmetric tensor completion was also studied, where Rao et al \cite{rao2015} used a variant of the Frank-Wolfe algorithm and Cai et al \cite{cai2019nonconvex} used two-stage non-convex optimization.
Another algorithm based on first-order optimization, and that uses integer optimization for a weak separation oracle was proposed in Bugg et al \cite{bugg2022nonnegative} for nonnegative tensors. Accelerated versions of this integer-optimization-based algorithm for nonnegative tensors were explored in Pan et al \cite{pan2023}.%A Frank-Wolfe variant algorithm using integer optimization is proposed for the completion of the nonnegative class of tensors \cite{bugg2022nonnegative}. Our approach is closely related to the last paper, where we modify the weak separation oracle to generalize the algorithm to all tensors.

\subsection{Contribution}
We design an algorithm for tensor completion of general tensors, which converges to a global optimum with a linear number of oracle calls and satisfies the information-theoretic rate. In past literature, tensor completion algorithms that perform both information-theoretically and computationally-well have been achieved only in special cases (some examples include nonnegative tensors, symmetric and orthogonal tensors). 

Our algorithm is a step in this direction for the case of general tensors. Building upon applications of various special-case tensor completion algorithms \cite{li2008image, aswani2016, liu2012tensor,zhang2019robust}, our algorithm for tensor completion opens up new applications that require positive \emph{and} negative entries. Examples of applications of tensor completion include social computing \cite{song2019tensor}, the moment method on multivariate distribution \cite{montanari2018spectral}, and healthcare applications \cite{gandy2011tensor, Bazerque_Mateos_Giannakis_2013}.

The main idea behind our algorithm is to define the tensor completion problem using a gauge norm, and then use a Frank-Wolfe-like first-order optimization algorithm to solve the newly defined convex optimization formulation. We define the gauge norm by constructing a convex polytope with its vertices as rank-1 tensors. The rank-1 tensor vertices help define the convex polytope using integer linear constraints. We relate the gauge norm to tensor rank, and analyze the norm's computational and statistical complexity to \textsc{np}-hard and low Rademacher complexity respectively.

Consequently, the tensor completion problem using gauge norm is also \textsc{np}-hard to solve to arbitrary accuracy. Nevertheless, since the formulation is a convex optimization, we design an algorithm using Blended Conditional Gradients (BCG) \cite{braun2019blended}. We construct a weak separation oracle for BCG using an integer linear optimization formulation, for which only a linear number of calls are needed. Integer optimization has a collection of mature, commercial-grade solvers that can navigate complex trade-offs between compute and solution quality on NP-hard problems. Moreover, weak separation enables us to deploy an additional heuristic to practically accelerate the computation of oracle calls, as well as to apply early termination in the integer programming solver when a sufficiently good solution has been found. Our numerical experiments demonstrate that this algorithm achieves the information-theoretic rate and is efficient for tensors with as large as ten million entries.

Our paper is organized as follows. In Section \ref{sec2}, we go over basic notations and the tensor completion problem. Section \ref{sec3} defines the gauge norm and establishes its relationship with tensor rank. We further analyze the norm complexity in Section \ref{sec4}. Then in Section \ref{sec5}, we discuss the complexity of the tensor completion problem with gauge norm, and design an algorithm. Finally, Section \ref{sec6} summarizes our numerical experiments, followed by the conclusion in Section \ref{sec7}.

\subsection{Going from Nonnegative to General Tensors}

Since the approach in Bugg et al \cite{bugg2022nonnegative} is closely related to this paper, here we provide a brief discussion about why going from nonnegative to general tensors requires the design and analysis of a new tensor completion algorithm:

The key idea behind Bugg et al \cite{bugg2022nonnegative} was to define a gauge norm using a $0$-$1$ polytope that represents the convex hull of all rank-1 nonnegative tensors whose maximum entry is 1. The polytope was designed in such a way that linear separation problems over the polytope could be written using integer linear constraints. The approach can be illustrated by an example: Consider the set $\{\zeta = \prod_{k=1}^p \nu_k : \nu_k\in\{0,1\}\}$ defined using a product, its linearization is given by $\{\zeta : 0\leq\zeta \leq \nu_k, \sum_k \nu_k + (1-p) \leq \zeta, \nu_k\in\{0,1\}\}$. A direct linearization is possible in this case because of the properties that: $\zeta = 0$ if any single $\nu_k = 0$, and $\zeta = 1$ if all $\nu_k = 1$.

The natural generalization of the idea behind Bugg et al \cite{bugg2022nonnegative}, which we pursue in this paper, is to design a gauge norm using a polytope whose vertices have $\pm 1$ entries  and that represents the convex hull of all (general) rank-1 tensors whose maximum/minimum entry is $\pm 1$. However, the above property used for linearization no longer holds when considering the new polytope. This can be shown via an example: For the set $\{\zeta = \prod_{k=1}^p \nu_k : \nu_k\in\{-1,+1\}\}$, a direct linearization is not possible because $\zeta = +1$ if an even number of $\nu_k = -1$, and $\zeta = -1$ if an odd number of $\nu_k = -1$. The generalization is fundamentally different because it involves the parity (i.e., odd or even) of the underlying items being multiplied. This is much more challenging and requires new computational design and theoretical analysis. 

%The complexity of the tensor completion problem with a gauge norm increases with general tensors. We see this reflected in our constraints used to model the gauge norm, where there is a linear increase (in tensor order) in the number of constraints defining a unit convex polytope for the norm. The reason for this is that the convex polytope defined using $\pm 1$ entries for general tensors has different properties than $0$-$1$ polytope for non-negative tensors. Consequently, we had to re-work the proofs for \textsc{np}-hardness and properties of our gauge norm.

General tensors also have well-posedness issues that do not occur for nonnegative tensors. In particular, the best low-rank approximation problem is not well-posed for tensors \cite{desilva2008}; this is related to a well-known example that a sequence of rank-2 tensors can be constructed whose limit has rank-3. In contrast, the best low-rank approximation problem \emph{is} well-posed for nonnegative tensors \cite{qi2016}. Given these phenomena for general tensors, there is a technical challenge in determining whether a new gauge norm as defined above can act as a convex surrogate for the rank of tensors.

%An order $p$ rank-1 tensor can be expressed as the tensor product of $p$ vectors. This implies that for a nonnegative rank-1 tensor that is a vertex of an appropriately-defined unit-ball, each entry in the tensor is equal to 0 when one or more of the corresponding entries in the $p$ vectors is 0, otherwise this tensor entry is $1$. This property is leveraged in \cite{bugg2022nonnegative} to construct their algorithm. However, such property does not hold for general tensors. Instead, an entry in a (general) rank-1 tensor that is a vertex of an appropriately-defined unit-ball is $-1$ when there is an odd number of negative values in the corresponding entries in the $p$ vectors and is $+1$ when there is an even number. Since techniques based on the property of nonnegative tensors will not apply to general tensors, this paper addresses this difference by modifying key components of the approach in \cite{bugg2022nonnegative} to construct a numerical algorithm for general tensor completion. In fact, this difference between nonnegative and general tensors makes (general) tensor completion computationally much more difficult than nonnegative tenor completion.

%Since tensors can be represented as the summation of rank-1 tensors by a \textsc{cp} decomposition, the substantial difference between nonnegative and general tensors in the rank-1 case would apply to all tensors.

\section{Preliminaries} \label{sec2}

By definition, a rank-1 tensor can be written as the tensor product of vectors, that is $\psi = \bigotimes_{k=1}^p \theta^{(k)}$, where $\theta^{(k)} \in \mathbb{R}^{r_k}$. Equivalently, each tensor entry is $\psi_x = \prod_{k=1}^p\theta^{(k)}_{x_k}$, where $\theta^{(k)}_{x_k}$ is the $x_k$-th element of vector $\theta^{(k)}$ for any index value $x_k \in [r_k]$.  When obvious, we will use $\theta_{x_k}$ instead of $\theta^{(k)}_{x_k}$. 

Let $\mathcal{B}_{\lambda}$ be the set of rank-1 tensors such that each entry of the tensor has absolute values less than or equal to $\lambda\in\mathbb{R}_+$:
\begin{equation}
\label{eq:blam}
\mathcal{B}_\lambda = \{\psi : \psi_x=\textstyle\lambda\cdot\prod_{k=1}^p\theta_{x_k}, \hspace{.1cm}
 \theta_{x_k}\in[-1,1] \text{ for }  x\in\mathcal{R}\},
\end{equation}
Then, the rank of a tensor is defined to be the minimum number of rank-1 tensors required to represent it. Formally, 
\begin{equation*}
\textstyle\rank(\psi) = \min \{q\ |\ \psi = \sum_{k=1}^q \psi^k, \psi^k \in \mathcal{B}_\infty \text{ for } k\in[q]\}.
\end{equation*}
Using this definition of tensor rank, a \textsc{cp} decomposition of the tensor is given by $\psi = \textstyle \sum_{k=1}^{\rank(\psi)} \psi^k$.

The tensor completion problem we consider begins with $n$ observations of the tensor, which are denoted by the pairs $(x\langle i\rangle, y\langle i\rangle) \in \mathcal{R}\times\mathbb{R}$ for $i \in [n]$. Here, $y\langle i \rangle$ is the (possibly noisy) observation of the tensor entry $\psi_{x\langle i\rangle}$. We note that the $x\langle i\rangle$ are assumed to be independent and identically distributed in our model, which means that any given entry of the tensor may be observed multiple times within the $n$ observations. Our approach is to solve the tensor completion problem using a least squares formulation:
\begin{equation}
\label{eq:tencompl}
\begin{aligned}
\widehat{\psi} \in \arg\min_{\psi}\ & \textstyle\frac{1}{n}\sum_{i=1}^n \big(y\langle i\rangle - \psi_{x\langle i\rangle}\big)^2\\
\mathrm{s.t.}\ & \|\psi\|_\pm\leq\lambda
\end{aligned}
\end{equation}
where the constraint uses a new norm $\|\psi\|_\pm$ that we design below. We show that our norm does in fact act as a convex surrogate for tensor rank, analyze its complexity, and demonstrate that it leads to a practical global algorithm for the tensor completion problem.

\section{Gauge Norm for Tensors} \label{sec3}
We construct a norm for general tensors using a gauge function \cite{chandrasekaran2012convex, Jaggi_2013,  bugg2022nonnegative}. To start with, consider rank-1 tensors. Let $\mathcal{S}_{\lambda}$ be the set of rank-1 tensors such that each entry of the tensor has an absolute value of some $\lambda \in \mathbb{R}_+$:
\begin{equation} \label{eq:slam}
\mathcal{S}_\lambda = \{\psi : \psi_x=\textstyle\lambda\cdot\prod_{k=1}^p\theta_{x_k}, \hspace{.1cm} \theta_{x_k}\in\{-1,1\} \text{ for }  x\in\mathcal{R}\}.
\end{equation}

Note that the difference between the sets $\mathcal{S}_{\lambda}$ and $\mathcal{B}_{\lambda}$ is that their rank-1 tensors have entries in $\{-\lambda,\lambda\}$ and $[-\lambda,\lambda]$ respectively. Our first step in defining a norm is to relate the convex hull of these sets of rank-1 tensors.

\begin{proposition}
\label{prop:convequiv}
The convex hulls of these sets are the same, meaning we have $\mathcal{C}_\lambda := \conv(\mathcal{B}_\lambda) = \conv(\mathcal{S}_\lambda)$.
\end{proposition}

\begin{proof}
The proof for this proposition is similar to the proof of Proposition 2.1 of \cite{bugg2022nonnegative}, but where a multilinear optimization problem is formulated by restricting the entries to $[-1,1]$ instead of $[0,1]$ as in (\ref{eqn:mop}).

The proof follows by showing two set inclusions $\conv(\mathcal{S}_\lambda) \subseteq \conv(\mathcal{B}_\lambda)$ and $\conv(\mathcal{B}_\lambda) \subseteq \conv(\mathcal{S}_\lambda)$. The first inclusion is immediate by definition since we have $\mathcal{S}_\lambda \subset \mathcal{B}_\lambda$. The second inclusion can be proved by contradiction. Suppose instead that $\conv(B_\lambda) \not\subset \conv(\mathcal{S}_\lambda)$. Then there exists a tensor $\psi' \in B_\lambda$ with $\psi' \not\in \conv(\mathcal{S}_\lambda)$. By the hyperplane separation theorem, there exists $\varphi \in \mathbb{R}^{r_1\times\cdots\times r_p}$ and $\delta > 0$ such that $\langle \varphi, \psi' \rangle \geq \langle \varphi, \psi\rangle + \delta$ for all $\psi \in \conv(\mathcal{S}_\lambda)$, where $\langle\cdot,\cdot\rangle$ is the usual inner product that is defined as the summation of elementwise multiplication. Now consider the multilinear optimization problem
\begin{equation}
\label{eqn:mop}
\begin{aligned}
\max\ & \langle \varphi, \psi \rangle\\
\mathrm{s.t.}\ & \psi_x=\textstyle\lambda\cdot\prod_{k=1}^p\theta_{x_k}, &\text{ for } x \in \mathcal{R}\\
& \theta_{x_k}\in[-1,1], &\text{ for }  x\in\mathcal{R}
\end{aligned}
\end{equation}
Proposition 2.1 of \cite{drenick1992multilinear} shows there exists a global optimum $\psi''$ of (\ref{eqn:mop}) with $\psi''\in \mathcal{S}_\lambda$. By construction, we must have $\langle \varphi, \psi'' \rangle \geq \langle \varphi, \psi' \rangle$, which implies $\langle \varphi, \psi'' \rangle \geq \langle \varphi, \psi\rangle + \delta$ for all $\psi \in \conv(\mathcal{S}_\lambda)$. But this last statement is a contradiction since $\psi'' \in \mathcal{S}_\lambda \subseteq\conv(\mathcal{S}_\lambda)$.
\end{proof}

\begin{remark}
Note $\mathcal{B}_\lambda = \lambda\mathcal{B}_1$, $\mathcal{S}_\lambda = \lambda\mathcal{S}_1$, and $\mathcal{C}_\lambda = \lambda\mathcal{C}_1$.
\end{remark}

\subsection{Defining the Norm}

$\mathcal{C}_\lambda$ is useful because it is a convex polytope with its vertices as the points in $\mathcal{S}_\lambda$. Our next step is to use $\mathcal{C}_\lambda$ to define a function that we prove in the next proposition is a gauge norm. The proof reveals the non-intuitive property that $\mathcal{C}_1$ has a non-empty interior and hence can be magnified (using $\lambda\mathcal{C}_1$) to cover all tensors. This property ensures that the below function is in fact a norm defined for all tensors.
\begin{proposition}
\label{prop:well-posed}
The function defined as
\begin{equation}
\|\psi\|_\pm := \inf \{\lambda \geq 0\ |\ \psi \in \lambda\mathcal{C}_1\}
\end{equation}
is a norm for all tensors $\psi \in \mathbb{R}^{r_1\times\cdots\times r_p}$.
\end{proposition}

\begin{proof}
We use the result from Example 3.50 of \cite{rockafellar2009} to conclude that $\|\cdot\|_\pm$ is a norm. To apply the result, we check that the required conditions on $\mathcal{C}_1$ hold.

By definition $\mathcal{C}_1$ is convex, closed, and bounded. $\mathcal{C}_1$ is also symmetric since for every $a \in \mathcal{C}_1$, $-a \in \mathcal{C}_1$ is also true. To see this, let $a = \sum_i \lambda_i \psi_i$ where $\psi_i \in \mathcal{S}_1$, $\lambda_i \in [0,1]$ and $\sum_i \lambda_i= 1$ and notice $-\psi_i \in \mathcal{S}_1$. Symmetry and convexity also ensure $0 \in \mathcal{C}_1$ since for any $a \in \mathcal{C}_1$, $\frac{1}{2}\left(a + (-a)\right) = 0$. 

The final, non-trivial condition required is that $\mathcal{C}_1$ has a non-empty interior. To prove this, we use Theorem 2.4 of \cite{Rockafellar_2015} that the dimension of $\mathcal{C}_1$ is the maximum of dimensions of simplices included in it. We construct a simplex in $\mathcal{C}_1$ that has dimension $\pi = \Pi_{i} r_i$, and hence $\mathcal{C}_1$ has a dimension of at least $\pi$. But the flattened vector space of tensors also has dimension $\pi$; consequently, the dimension of $\mathcal{C}_1$ is at most $\pi$. Hence, $\mathcal{C}_1$ has a dimension $\pi$, which is the full dimension of the space, implying that the set $\mathcal{C}_1$ must have a non-empty interior. 

The rest of the proof constructs such a simplex of dimension $\pi$. Consider a polytope $D = \conv(0 \: \cup \: \{d^x\}_{x \in \mathcal{R}})$. Here, for any $x =(x_1,\cdots,x_p) \in \mathcal{R}$, consider $d^x = \bigotimes_{k=1}^p \beta^{x_k}$ with each vector $\beta^{x_k} \in \mathbb{R}^{r_k}$ defined as
\[\beta^{x_k} = \begin{cases}
  \mathbbm{1} & \text{if $x_k=1$} \\
  \mathbbm{f}_{x_k} & \text{if $x_k \neq 1$}
\end{cases},\]
where $\mathbbm{1}$ is a vector of one's and $\mathbbm{f}_j$ is a vector with $-1$ in position $j$ and one's elsewhere. One can verify that $\{\beta^{x_k}\}_{x_k \in [r_k]}$ are linearly independent vectors and make a complete basis for $\mathbb{R}^{r_k}$. Since the tensor product of linearly independent vectors gives linearly independent tensors, the tensors $\{d^x\}_{x \in \mathcal{R}}$ are all linearly independent. Consequently, $\{d^x - 0\}_{x \in \mathcal{R}}$ are linearly independent and $\{d^x\}_{x\in \mathcal{R}} \cup \: 0$ are affinely independent. By definition, the polytope $D$, which is a convex hull of $|\mathcal{R}| +1 = \pi +1$ affinely independent points, is a simplex of dimension $\pi$. Note that all the points are in $\mathcal{C}_1$, and hence so is the simplex.

With this, $\mathcal{C}_1$ is shown to satisfy all the required conditions, which means that the proposition holds.
\end{proof}

Towards our end goal of developing a practical algorithm for tensor completion using this gauge norm, our next result provides an alternative characterization of the vertices of $\mathcal{C}_\lambda$. This result is important because it shows that these vertices can be represented by linear inequality constraints using (binary) integer variables. This is unlike, say, the nuclear norm, which under the $\ell_2$ norm defines an inherently nonlinear (continuous) feasible region.

\begin{proposition}\label{prop:bilp} Consider the set defined as
\begin{equation*}
    \begin{aligned}
\widehat{\mathcal{S}}_\lambda = \big\{&\psi :\;  \psi_x = \lambda\cdot y_{x,1}   & x \in \mathcal{R}\\
& y_{x,k} \geq (-\theta_{x_k}-y_{x,k+1} -1) & k \in [p-1], x \in \mathcal{R}\\
& y_{x,k} \geq (\theta_{x_k}+y_{x,k+1} -1) & k \in [p-1], x \in \mathcal{R}\\
& y_{x,k} \leq  (\theta_{x_k}-y_{x,k+1} +1)& k \in [p-1], x \in \mathcal{R}\\
& y_{x,k} \leq  (-\theta_{x_k}+y_{x,k+1} +1) & k \in [p-1], x \in \mathcal{R}\\
&y_{x,p} = \theta_{x_p} & x \in \mathcal{R}\\
& \theta_{x_k} \in \{-1,1\} & x \in \mathcal{R}\\ 
& \theta^k \in \mathbb{R}^{r_k},\;  y_{x,k} \in \mathbb{R} & k\in[p], x \in \mathcal{R}&\big\}.
\end{aligned}
\end{equation*}
We have that $\widehat{\mathcal{S}}_\lambda = \mathcal{S}_\lambda$.

\end{proposition}
\begin{proof}
    We show that the constraints defining the sets $\mathcal{S}_\lambda$ and $\widehat{\mathcal{S}}_\lambda$ are equivalent. Consider some $x \in \mathcal{R}$. From the definition of $\mathcal{S}_\lambda$ in (\ref{eq:slam}), we have $\psi_x = \lambda \prod_{i=1}^p \theta_{x_k}$. Define $y_{x,k} = \prod_{i=k}^{p} \theta_{x_k}$, for $k \in [p]$ so that $\psi_x = \lambda y_{x,1}$. Or equivalently, in a recursive relationship, $y_{x,k} = \theta_{x_k} y_{x,k+1}$ for $k \in [p-1]$ and $y_{x,p} = \theta_{x_p}$. The recursive constraints can be thought of as a negated-XOR relation, and linearized by transformations for conjunctive and disjunctive statements (see Section 2.5 of \cite{Conforti_Cornuéjols_Zambelli_2014}). These linearized constraints correspond to constraints 2-5 in the definition of $\widehat{\mathcal{S}}_\lambda$ as given above. This shows that for each $x \in \mathcal{R}$, the tensor $\psi_x$ is defined the same in both $\mathcal{S}_\lambda$ and $\widehat{\mathcal{S}}_\lambda$.
\end{proof} 

\subsection{Relation between Gauge Norm and Tensor Rank}
We next establish a relationship between tensor rank and our gauge norm, and we use this to argue that our norm is a meaningful constraint in the tensor completion problem. The underlying issue is related to the fact that the best low-rank approximation problem is not well-posed for tensors \cite{desilva2008}, which is in sharp contrast to the case of nonnegative tensors for which the best low-rank approximation problem is well-posed \cite{qi2016}.

For the results in this subsection, we impose a regularity condition to eliminate such pathological behavior of tensors. 

\begin{assumption}[Regularity Condition]
    \label{regcondn}
\textit{ Consider a class of tensors defined by the set}
\begin{multline}
    \label{RC}
\Gamma = \big\{\psi : \exists \text{ \textsc{cp} decomposition of } \psi \text{ with terms } \psi^k \text{ s.t. } \\
\|\psi^k\|_{\max} \leq \|\psi\|_{\max} \text{ for } k \in [\mathrm{rank}(\psi)]\big\}.
%\Gamma = \{\psi : \|\psi^k\|_{\max} \leq \|\psi\|_{\max} \text{ for } k \in [\mathrm{rank}(\psi)],\\
%\text{where } \psi^k \text{ form a \textsc{cp} decomposition of } \psi\}.
\end{multline}
\textit{This class is such that each tensor $\psi$ has its largest entry at least as large as the largest entry of each \textsc{cp} term $\psi^k$.}
\end{assumption}

\begin{remark}
    A \textsc{cp} decomposition always exists for a finite-valued tensor, but it may not be unique. The class defined above asks that the regularity condition holds for at least one \textsc{cp} decomposition, but does not make any statement about holding for all the possible \textsc{cp} decompositions.  
\end{remark}

The following proposition suggests that the norm $\|\psi\|_\pm$, which is convex, can be a useful alternative to tensor rank. 
 
\begin{proposition}
\label{prop:ranks}
For any $\psi \in \Gamma$ that satisfies Assumption \ref{regcondn}, we have $\|\psi\|_{\max} \leq \|\psi\|_\pm \leq \rank(\psi)\cdot\|\psi\|_{\max}$.
\end{proposition}
\begin{proof}
The right-side inequality requires the regularity assumption. Using the \textsc{cp} decomposition and the triangle inequality for norms,
\begin{equation*}
\label{eqn:propeqnproof}
\begin{split}
\|\psi\|_\pm\leq \textstyle\sum_{k=1}^{\rank(\psi)}\|\psi^k\|_\pm = \textstyle\sum_{k=1}^{\rank(\psi)}\|\psi^k\|_{\max}, \\
\text{ where $\psi^k\in\mathcal{B}_\infty$}
\end{split}
\end{equation*}
The last equality follows from $\|\psi^k\|_{\pm} = \|\psi^k\|_{\max}$ when $\psi^k\in\mathcal{B}_\infty$. Using
$\|\psi^k\|_{\max} \leq \|\psi\|_{\max}$ from Assumption \ref{regcondn} gives the right-side inequality.

The proof of the left-side inequality is similar to Proposition 2.4 of \cite{bugg2022nonnegative}. For any $\lambda \geq 0$, if $\psi \in\mathcal{S}_{\lambda}$, then by definition $\|\psi\|_{\max} = \lambda$. By the convexity of norms we have that: if $\psi \in \mathcal{C}_{\lambda}$, then $\|\psi\|_{\max} \leq \lambda$. This means that $\forall \: \lambda \geq 0$ we have $\mathcal{C}_{\lambda} \subseteq \mathcal{U}_{\lambda} := \{\psi : \|\psi\|_{\max} \leq \lambda\}$, and thus $\inf \{\lambda \ |\ \psi \in \lambda\mathcal{U}_1\} \leq \inf \{\lambda \ |\ \psi \in \lambda\mathcal{C}_1\}$. But this is equivalent to $\|\psi\|_{\max} \leq \|\psi\|_{\pm}$, which proves the left-side inequality.
\end{proof}

\begin{comment}
\begin{remark}
%These bounds are tight. The lower and upper bounds are achieved by all nonnegative rank-1 tensors. As another example, the identity matrix with $k$ columns achieves the upper bound with $\rank_+(\psi) = k$.
These bounds are tight. The lower and upper bounds are achieved by all rank-1 tensors. The identity matrix with $k$ columns achieves the upper bound with $\rank(\psi) = k$.
\end{remark}
\end{comment}

\section{Complexity Analysis of Norm} \label{sec4}
We show that calculating the norm $\|\ \cdot\|_\pm$ is \textsc{np}-hard. Despite this, it is still useful for tensor completion because it is defined using a convex polytope $\mathcal{C}_1$ whose vertices admit a convenient representation, as described in Proposition \ref{prop:bilp}, that we will use to develop a practical algorithm. 
\begin{proposition}[\textbf{Computational complexity}]
\label{prop:nphard}
The norm $\|\cdot\|_\pm$ is \textsc{np}-hard to approximate to arbitrary accuracy.
\end{proposition}

\begin{proof}
Note that $\|\varphi\|_{\circ} = \sup\{|\langle \varphi, \psi\rangle| \ |\ \|\psi\|_\pm \leq 1\} = \sup\{\langle \varphi, \psi\rangle\ |\ \psi \in \mathcal{C}_1\}$ is the dual norm for $\|\cdot\|_\pm$. The approximation of $\|\cdot\|_{\circ}$ can be reduced to approximation of the norm $\|\cdot\|_\pm$ in polynomial time from theorems 3 and 10 of \cite{friedland2016computational}. Our main idea is to give a polynomial-time reduction of an \textsc{np}-Hard problem to an approximation of $\|\cdot\|_{\circ}$. In particular, we prove that calculating the $\infty,1$ subordinate matrix norm is polynomial-time reducible to $\sup\{\langle \varphi, \psi\rangle\ |\ \psi \in \mathcal{S}_1\}$. Without loss of generality, assume $p=2$ and $d:= r_1 = r_2$.

The decision version of $\sup\{\langle \varphi, \psi\rangle\ |\ \psi \in \mathcal{S}_1\}$ is:

\textit{Question:} Does there exist $\theta_{x_k} \in \{-1,1\}$ for all $x_k \in [d]$ with $k=1,2$ such that for a given $L$ we have
\begin{equation*}
\textstyle\sum_{x_1 = 1}^d\sum_{x_2=1}^d \varphi_{x_1x_2} \theta_{x_1}\theta_{x_2}  \geq L \;?    
\end{equation*}
From Proposition 1 of \cite{Rohn_2000}, we have $\|W\|_{\infty,1} = \{\max \sum_{i,j} W_{ij}x_iy_j \:|\: x_i,y_j \in \{-1,1\}\}$ for a matrix $W$. The decision version of the $\infty,1$ subordinate matrix norm for the special case of $M$-matrices can be written as:

\textit{Question:} Does there exist $x_i,y_i \in \{-1,1\}$ for all $i \in [d]$ such that for a given $L' \geq 0$ and a \textit{symmetric, positive definite} matrix $W \in \mathbb{R}^{d \times d}$ satisfying $w_{ij} \leq 0$ for all $i \neq j$ and
\begin{equation*}
    \textstyle\sum_{i= 1}^d \sum_{j=1}^d w_{ij} x_iy_j \geq L' \;?
\end{equation*}
Clearly, setting $L =L'$ and $\varphi = W$ is a valid polynomial time reduction. Since $\mathcal{C}_1 = \conv(\mathcal{S}_1)$ and $\langle \varphi, \psi\rangle$ is linear, $\|\varphi\|_{\circ} = \sup\{\langle \varphi, \psi\rangle\ |\ \psi \in \mathcal{S}_1\}$. The result now follows since approximately solving any $\infty,p$ subordinate matrix norm, where $p \in [1,\infty)$, to arbitrary accuracy is \textsc{np}-hard \cite{Hendrickx_Olshevsky_2010}. In particular, Theorem 5 of \cite{Rohn_2000} shows that it is an \textsc{np}-hard problem for $M$-matrices.
\end{proof}

The same proof of Corollary 3.2 in \cite{bugg2022nonnegative} combined with the above result establishes \textsc{np}-completeness:

\begin{corollary}
[\cite{bugg2022nonnegative}]\label{corollary:npcnorm} 
Given $K \in\mathbb{R}_+$ and $\psi \in \mathbb{R}^{r_1\times\cdots\times r_p}$, it is \textsc{np}-complete to determine if $\|\psi\|_\pm \leq K$.
\end{corollary}

Rademacher complexity, from computational learning theory, is used to characterize the richness of a class of functions \cite{bartlett2002,srebro2010}. Roughly speaking, function classes with lower Rademacher complexity can be learned using less samples. From the following proposition, one can check that the norm $\|\ \cdot\|_\pm$ has an exponentially smaller Rademacher complexity than the max and Frobenius norms for tensors.
\begin{comment}
We next show that our norm $\|\psi\|_\pm$ has low stochastic complexity. Let $X = \{x\langle 1\rangle, \ldots, x\langle n\rangle\}$, and suppose $\sigma_i$ are independent and identically distributed (i.i.d.)  Rademacher random variables (i.e., $\sigma_i = \pm 1$ with probability $\frac{1}{2}$) \cite{bartlett2002,srebro2010}. The Rademacher complexity for a set of functions $\mathcal{H}$ is   $\mathsf{R}(\mathcal{H}) = \mathbb{E}(\sup_{h\in\mathcal{H}}\textstyle\frac{1}{n}|\sum_{i=1}^n\sigma_i\cdot h(x\langle i\rangle)|)$, and the \emph{worst case} Rademacher complexity of $\mathcal{H}$ is $\mathsf{W}(\mathcal{H}) = \sup_{X} \mathbb{E}_\sigma(\sup_{h\in\mathcal{H}}\textstyle\frac{1}{n}|\sum_{i=1}^n\sigma_i\cdot h(x\langle i\rangle)|)$. These notions can be used to measure the complexity of sets of matrices \cite{srebro2005rank} or tensors \cite{aswani2016} through interpreting each tensor as a function $\psi: \mathcal{R} \rightarrow \mathbb{R}$ from a set of indices $x \in \mathcal{R}$ to the corresponding entry of the tensor $\psi_{x}$. This complexity notion is useful for the completion problem because it can be directly translated into generalization bounds.
\end{comment}
\begin{proposition} 
[\textbf{Stochastic Complexity}]\label{Rcompl}
We have $\mathsf{R}(\mathcal{C}_\lambda) \leq \mathsf{W}(\mathcal{C}_\lambda) \leq 2\lambda\sqrt{\rho/n}$, where $\mathsf{R}(\cdot)$ and $\mathsf{W}(\cdot)$ are the Rademacher and worst case Rademacher complexities. 
\end{proposition}
\begin{proof}
The proof is similar to Proposition 3.3 of \cite{bugg2022nonnegative}. By definition, $\mathsf{R}(\mathcal{C}_\lambda) = \mathbb{E}_{\sigma}(\sup_{\Psi\in\mathcal{C}_\lambda}\textstyle\frac{1}{n}|\sum_{i=1}^n\sigma_i\cdot \psi_{x\langle i\rangle}|)$ and $\mathsf{W}(\mathcal{C}_\lambda) = \sup_{X} \mathsf{R}(\mathcal{C}_\lambda)$, where $\sigma_i$ are independent and identically distributed (i.i.d.)  Rademacher random variables (i.e., $\sigma_i = \pm 1$ with probability $\frac{1}{2}$) \cite{bartlett2002,srebro2010}. Hence, $\mathsf{R}(\mathcal{C}_\lambda) \leq \mathsf{W}(\mathcal{C}_\lambda)$. Recall that $\mathcal{C}_{\lambda} = \conv(\mathcal{S}_\lambda)$ by Proposition \ref{prop:convequiv}. This means $\mathsf{W}(\mathcal{C}_\lambda) = \mathsf{W}(\mathcal{S}_\lambda)$ \cite{ledoux1991,bartlett2002}. Next, observe that
\begin{equation}
\begin{aligned}
\mathsf{W}(\mathcal{C}_\lambda) &=  \mathsf{W}(\mathcal{S}_\lambda) \textstyle=\sup_X\mathbb{E}_\sigma\big(\sup_{\psi\in\mathcal{S}_\lambda}\textstyle\frac{1}{n}\big|\sum_{i=1}^n\sigma_i\cdot \psi_{x\langle i\rangle}\big|\big)\\
&\textstyle=\sup_X\mathbb{E}_\sigma\big(\max_{\psi\in\mathcal{S}_{\lambda}}\textstyle\frac{1}{n}\cdot\sum_{i=1}^n\sigma_i\cdot \psi_{x\langle i\rangle}\big) \leq \sup_X r\sqrt{2\log \#\mathcal{S}_\lambda}/n
%\mathsf{W}(\mathcal{C}_\lambda) =  \mathsf{W}(\mathcal{S}_\lambda) &=\sup_X\mathbb{E}_\sigma\big(\sup_{\psi\in\mathcal{S}_\lambda}\textstyle\frac{1}{n}\big|\sum_{i=1}^n\sigma_i\cdot \psi_{x\langle i\rangle}\big|\big)\\
%&=\sup_X\mathbb{E}_\sigma\Big(\max_{\psi\in\mathcal{P}_{\lambda}}\textstyle\frac{1}{n}\cdot\sum_{i=1}^n\sigma_i\cdot \psi_{x\langle i\rangle}\Big) \leq \sup_X r\sqrt{2\log \#\mathcal{P}_\lambda}/n
\end{aligned}
\end{equation}
where in the second line since the set $\mathcal{S}_\lambda$ is finite, we used the Finite Class Lemma \cite{massart2000} with $r = \max_{\psi\in\mathcal{S}_\lambda}\textstyle\sqrt{\sum_{i=1}^n(\psi_{x\langle i\rangle})^2} \leq \lambda\sqrt{n}$, and replaced the supremum with a maximum. This inequality on $r$ is because $\mathcal{S}_\lambda$ consists of tensors whose entries are from $\{-\lambda,\lambda\}$. Thus $\mathsf{W}(\mathcal{C}_\lambda) \leq \lambda\sqrt{(2\log2)\cdot \rho/n} \leq \lambda\cdot 2\sqrt{\rho/n}$. 
\end{proof}

\section{Algorithm for Tensor Completion}\label{sec5}
We now turn our attention towards numerical solution of the tensor completion problem (\ref{eq:tencompl}) using our norm $\|\psi\|_\pm$.

\subsection{Complexity Analysis of Tensor Completion}
By interpreting the tensor completion problem (\ref{eq:tencompl}) as a \emph{convex aggregation} problem \cite{nemirovski2000topics,tsybakov2003optimal,lecue2013empirical} for a finite set of functions, one can arrive at the generalization bound for the solution. Interestingly, these generalization bounds are the same in the special case of nonnegative tensors \cite{bugg2022nonnegative}. We believe this is because the proof for nonnegative tensors did not exploit the non-negativity, and hence could have led to non-tight generalization bounds. For completeness, we state these statistical guarantees in the following two results:

\begin{proposition} [\cite{lecue2013empirical}]
Suppose $|y| \leq b$ almost surely. Given any $\delta > 0$, with probability at least $1-4\delta$ we have that
\begin{multline}
\mathbb{E}\big((y - \psi_x)^2\big) \leq \min_{\varphi \in \mathcal{C}_\lambda} \mathbb{E}\big((y-\varphi_x)^2\big) +\\
c_0\cdot \max\big[b^2,\lambda^2\big]\cdot \max\big[\zeta_n, \textstyle\frac{\log(1/\delta)}{n}\big],
  \end{multline}
where $c_0$ is an absolute constant and 
\begin{equation}
\label{eqn:zn}
    \zeta_n = \begin{cases} \frac{2^\rho}{n}, & \text{if } 2^\rho \leq \sqrt{n}\\
    \sqrt{\frac{1}{n}\log\Big(\frac{e2^\rho}{\sqrt{n}}\Big)}, & \text{if } 2^\rho > \sqrt{n}\end{cases}
\end{equation}
\end{proposition}
\begin{comment}
\begin{remark} Note that $\zeta_n = o(\sqrt{\rho/n})$, and in some regimes $\zeta_n$ can be considerably faster than the $\sqrt{\rho/n}$ rate.
\end{remark}
Generalization bounds under specific noise models, such as an additive noise model, follow as a corollary to the above proposition combined with Proposition \ref{prop:ranks}. 
\end{comment}
\begin{corollary}[\cite{bugg2022nonnegative}]
\label{cor:mtrs}
Suppose $\psi\in\Gamma$ is a tensor (satisfying Assumption \ref{regcondn}) with $\rank(\psi) = k$ and $\|\psi\|_{\max} \leq \mu$. Under an additive noise model, if $(x\langle i \rangle, y\langle i \rangle)$ are independent and identically distributed with $|y\langle i \rangle - \varphi_{x\langle i\rangle}| \leq e$ almost surely and $\mathbb{E}y\langle i\rangle = \psi_{x\langle i\rangle}$. Then given any $\delta > 0$, with probability at least $1-4\delta$ we have 
\begin{equation}
\mathbb{E}\big((y - \widehat{\psi}_x)^2\big) \leq e^2 + c_0\cdot \big(\mu k + e)^2\cdot \max\big[\zeta_n, \textstyle\frac{\log(1/\delta)}{n}\big],
\end{equation}
where $\zeta_n$ is as in (\ref{eqn:zn}) and $c_0$ is an absolute constant.
\end{corollary}

\begin{remark}
The above result achieves the information-theoretic rate when the rank $k = O(1)$.
\end{remark}

Unsurprisingly, the tensor completion problem (\ref{eq:tencompl}) is \textsc{np}-hard, because approximating the norm $\|\cdot\|_\pm$ is \textsc{np}-hard and there is a polynomial-time reduction of the problem (\ref{eq:tencompl}) to the \textsc{np}-hard weak membership problem \cite{bugg2022nonnegative}.

\begin{proposition}
\label{prop:nntnph}
The tensor completion problem (\ref{eq:tencompl}) is \textsc{np}-hard to solve to arbitrary accuracy. Also, its decision version is \textsc{np}-complete.
\end{proposition}

\begin{proof}
The proof for this proposition is similar to the proof of Proposition 4.4 of \cite{bugg2022nonnegative}. Define the ball of radius $\delta > 0$ centered at a tensor $\psi$ to be $B(\psi,\delta) = \{\varphi : \|\varphi - \psi\|_F \leq \delta\}$. Next define $W(\mathcal{C}_1,\delta) = \bigcup_{\psi\in\mathcal{C}_1} B(\psi,\delta)$ and $W(\mathcal{C}_1,-\delta) = \{\psi \in \mathcal{C}_1 : B(\psi,\delta) \subseteq\mathcal{C}_1\}$. The weak membership problem for $\mathcal{C}_1$ is that given a tensor $\psi$ and a $\delta > 0$ decide whether $\psi \in W(\mathcal{C}_1,\delta)$ or $\psi\notin W(\mathcal{C}_1,-\delta)$. Theorem 10 of \cite{friedland2016computational} shows that approximation of the norm $\|\cdot\|_{\pm}$ is polynomial-time reducible to the weak membership problem for $\mathcal{C}_1$. Since Proposition \ref{prop:nphard} shows that approximation of the norm $\|\cdot\|_{\pm}$ is NP-hard, the result follows if we can reduce the weak membership problem to (\ref{eq:tencompl}). 

Suppose we are given inputs $\psi$ and $\delta$ for the weak membership problem. Choose $x\langle i\rangle$ for $i = 1,\ldots,\pi$ such that each element in $\mathcal{R}$ is enumerated exactly once. Next choose $y\langle i\rangle = \psi_{x\langle i\rangle}$ and $\lambda = 1$. Finally, note if we solve (\ref{eq:tencompl}) and the minimum objective value is less than or equal to $\delta$, then we have $\psi \in W(\mathcal{C}_1,\delta)$; otherwise, we have $\psi \notin W(\mathcal{C}_1,-\delta)$. The result now follows since this was the desired reduction.
\end{proof}

\subsection{Numerical Algorithm}

The problem (\ref{eq:tencompl}) is a convex optimization formulation despite being \textsc{np}-hard. This property enables the application of various first-order convex optimization algorithms. In particular, we use a variant of the Frank-Wolfe algorithm called \emph{Blended Conditional Gradients} (BCG) \cite{braun2019blended}.  The main alternative approaches for iterative optimization seem to run into issues inherent to the implicit description of the feasible region; for instance, we do not have an effective barrier function available for interior point methods, nor do we have an efficient projection oracle that can be leveraged in projected gradient descent.

The BCG algorithm makes subsequent iterate updates based on the gradient of the objective function at the current iterate. It needs a weak separation oracle
%, as detailed in \ref{alg:wso}, 
to find a new vertex that reduces a gradient-based linear objective. The weak separation oracle is given in Algorithm \ref{alg:wso} \cite{bugg2022nonnegative}. The output of the oracle should either give a vertex that accomplishes separation, or a certificate that separation is not possible. We design our weak separation oracle using two algorithms, an integer optimization problem in (\ref{eqn:intprog}) and an alternating maximization heuristic. 

Since integer optimization is likely to be more computationally expensive than the alternating heuristic, we first try the latter several times with different randomized initializations. The heuristic adapts Algorithm \ref{alg:am}, where solutions are explored by toggling between $\pm 1$ \cite{bugg2022nonnegative}.  The following (nonlinear) separation objective is minimized over binary $\theta$, given a BCG-generated parameter $c$:

\begin{equation}
\label{eq:multobj}
z_M(\theta) := \sum_{x \in \mathcal{R}} \langle c_x,\psi_x- \textstyle\lambda\cdot\prod_{k=1}^p\theta_{x_k}\rangle.
\end{equation}

\begin{algorithm}[t]
   \caption{Weak Separation Oracle for $\mathcal{C}_\lambda$}
   \label{alg:wso}
\begin{algorithmic}
   \STATE {\bfseries Input:} linear objective $c \in \mathbb{R}^{r_1\times\cdots\times r_p}$, point $\psi\in\mathcal{C}_\lambda$, accuracy $K\geq 1$, gap estimate $\Phi > 0$, norm bound $\lambda$
   \STATE {\bfseries Output:} Either (1) vertex $\varphi\in\mathcal{S}_\lambda$ with $\langle c, \psi-\varphi\rangle \geq \Phi/K$, or (2) {\bfseries false:}  \\ $\langle c, \psi-\varphi\rangle \leq \Phi$ for all $\varphi\in\mathcal{C}_\lambda$ 
\end{algorithmic}
\end{algorithm}

\begin{algorithm}[t]
   \caption{Alternating Maximization}
   \label{alg:am}
\begin{algorithmic}
   \STATE {\bfseries Input:} linear objective $c \in \mathbb{R}^{r_1\times\cdots\times r_p}$, point $\psi\in\mathcal{C}_\lambda$, norm bound $\lambda$, incumbent \\ solution $\hat \theta \in \mathcal{S}_\lambda$, objective function $\textstyle z_M(\theta) := \sum_{x \in \mathcal{R}} \langle c_x,\psi_x- \textstyle\lambda\cdot\prod_{k=1}^p\theta_{x_k}\rangle$.
    \STATE {\bfseries Output:} Best known solution $\theta$
\vspace{.1cm}
   \STATE $\theta \gets \hat \theta$
   \STATE $z \gets z_M(\theta)$
   \FOR{$i=1$ {\bfseries to} $p$}
        \FOR{$k=1$ {\bfseries to} $r_{i}$}
            \STATE $\theta^{(i)}_{k} \gets -1 \cdot \theta^{(i)}_{k}$
            \IF{$z_{M} (\theta) > z$}
                \STATE  $z \gets z_M(\theta)$
            \ELSE
                \STATE{$\theta^{(i)}_{k} \gets -1 \cdot \theta^{(i)}_{k}$}
            \ENDIF
        \ENDFOR
   \ENDFOR
\end{algorithmic}
\end{algorithm}

The heuristic considers one dimension at a time, setting the corresponding entries of $\theta$ to optimize within the local neighbourhood defined by the given dimension (which can be done by simply considering the signs of $c$).  It runs in polynomial time (linear in the size of $\theta$) and has been seen to speed up the computation in our simulations by reducing calls to the integer optimization solver. However, note that the heuristic is merely that and cannot in general give a certificate for the non-existence of separation, which requires global optimization. Indeed, even in the case of a matrix ($p=2$), $z_M$ represents a generic form of NP-hard Quadratic Unconstrained Binary Optimization (QUBO) (see, e.g. \cite{lewis2017quadratic}).

If the heuristic is unable to yield a separating cut, we implement the following integer optimization problem. Note that $\langle \cdot, \cdot\rangle$ is the dot product of tensors obtained by flattening them into vectors.
\begin{equation}
\label{eqn:intprog}
\begin{aligned}
\max_{\varphi,\theta} & \langle c, \psi-\varphi\rangle\\  
 \mbox{s.t. } & \varphi_x = \lambda y_{1,k} & x\in\mathcal{R} \\
 & y_{x,k} \geq (-\theta_{x_k}-y_{x,k+1} -1) & k \in [p-1], x\in\mathcal{R} \\
&  y_{x,k} \geq (\theta_{x_k}+y_{x,k+1} -1) & k \in [p-1], x\in\mathcal{R}\\
& y_{x,k} \leq  (\theta_{x_k}-y_{x,k+1} +1) & k \in [p-1], x\in\mathcal{R}\\
&  y_{x,k} \leq  (-\theta_{x_k}+y_{x,k+1} +1) & k \in [p-1], x\in\mathcal{R}\\
&y_{x,p} = \theta_{x_p} & x\in\mathcal{R} \\
&\theta_{x_k} \in \{-1,1\} & k\in[p], x\in\mathcal{R}
%  \mbox{s.t. } &\lambda\cdot(1-p)+\textstyle\lambda\cdot\sum_{k=1}^p \theta_{x_k} \leq \varphi_x & x\in\mathcal{R}\\
% & 0\leq\varphi_x\leq \lambda\cdot\theta_{x_k} &  k\in[p], x\in\mathcal{R} \\ 
% &\theta_{x_k} \in \{0,1\} &  k\in[p], x\in\mathcal{R}
\end{aligned}
\end{equation}
Note that, in principle, for some binary variable $q\in \{-1,1\}$, one can directly branch on the disjunction:
\[q = -1 \lor q = 1.\]
However, in implementation we choose instead to reformulate via 0-1 auxiliary variables, $\sigma := (q+1)/2$, as enforcing $\sigma \in \{0,1\}$ obviates the need for imposing $q \in \{-1,1\}$. Integer optimization solvers typically can only handle 0-1 binary variables by default instead of $\{-1,1\}$ due to the more extensive set of techniques developed for 0-1 variables.

%$\theta_{x_k} \in \{-1,1\}$ can be represented as
%\[\theta_{x_k} \leq -1\ \lor\ \theta_{x_k} \geq 1, x_k\in\{-1, 0, 1\}.\]
%We linearized the two disjunctive constraints with Big-M method:
%\[\theta_{x_k} \leq -1 + M\sigma_{x_k}, \theta_{x_k} \geq 1 - M(1-\sigma_{x_k}), \sigma_{x_k}\in \{0, 1\}.\] 
%Since $\theta_{x_k}\in [-1, 1]$, we set $M = 2$. Then
%\[\theta_{x_k} \leq -1 + 2\sigma_{x_k}, \theta_{x_k} \geq -1 + 2\sigma_{x_k}), \sigma_{x_k}\in \{0, 1\},\]
%which implies $\sigma_{x_k} = (\theta_{x_k}+1)/2$. So we introduce auxiliary binary variables $\sigma_{x_k}$ and constraints $\sigma_{x_k} = (\theta_{x_k}+1)/2$ for $k\in[p]$ and $x\in\mathcal{R}$ to Eq.~\ref{eqn:intprog}.

Since we are looking for weak separation, the integer optimization solver is made to terminate when a solution with an objective greater than $\Phi/K$ is found. In case of no such solution, the dual bound $z$ from the solver serves as a no-separation certificate satisfying $\langle c, \psi-\varphi\rangle \leq z \leq \Phi$.

\begin{proposition}
    Using BCG approach, the convex optimization formulation of the tensor completion problem in (\ref{eq:tencompl}) can be solved in a linear number of calls to the Weak Separation Oracle in Algorithm \ref{alg:wso}.
\end{proposition}
\begin{proof}
    The linear convergence of BCG is guaranteed when the feasible set is a polytope and the objective function is strictly convex. Note that the feasible set $\|\psi\|_\pm\leq\lambda$ in Problem (\ref{eq:tencompl}) is a polytope. Further, it is easy to construct a strictly convex problem by projecting the feasible space onto the set of unique, observed tensor entries \cite{bugg2022nonnegative}.  Specifically, we use the equivalent reformulation in which we change the feasible set from $\{\psi : \|\psi\|_\pm\leq\lambda\} = C_\lambda$ to $\mathrm{Proj}_U(\mathcal{C}_\lambda)$ where the projection is done over the unique indices specified by the set $U$. Since the conditions for linear convergence are satisfied by our problem, our algorithm terminates in a linear number of oracle calls to Algorithm \ref{alg:wso}.
\end{proof} 

In fact, we find the BCG algorithm to be practically efficient for our problem since the weak linear separation oracle can accept (sufficiently good) suboptimal solutions that may be found at early termination of a global solver. When we design our weak separation oracle calls to an integer optimization problem, we explicitly define a tolerance for early termination. This works out well as integer optimization solvers usually discover near-optimal solutions fast and subsequently work hard to certify them.

\begin{figure*}
  \centering
  \begin{subfigure}{}
    \includegraphics[width=0.48\linewidth]{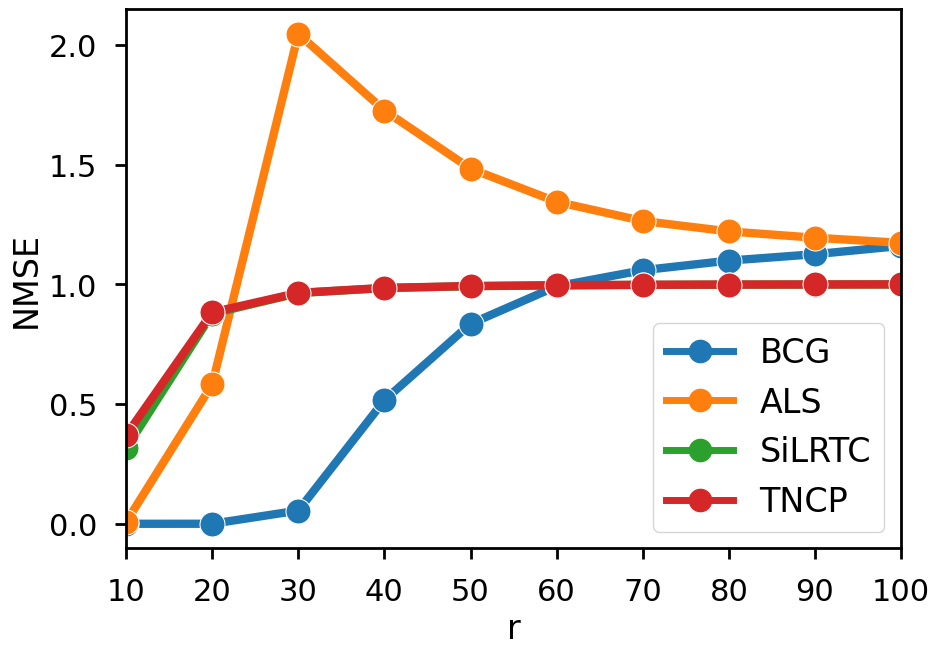}
 %   \label{fig:sub1a}
  \end{subfigure}
 % \hfill
  \begin{subfigure}{}
    \includegraphics[width=0.48\linewidth]{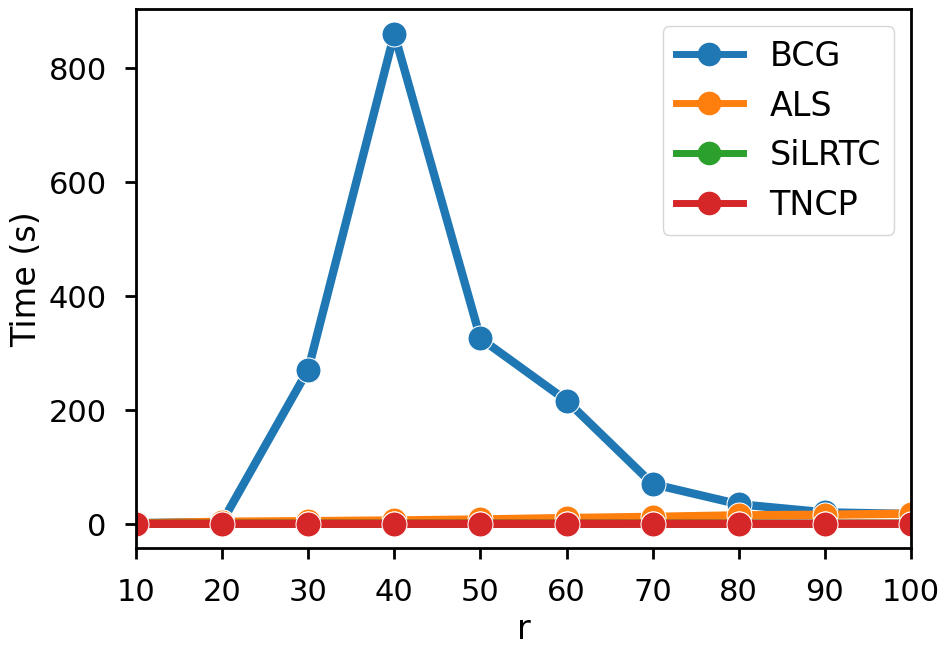}
%    \label{fig:sub1b}
  \end{subfigure}
  \caption{NMSE and computation time (in s) for order-3 tensors with size $r \times r \times r$ and $n = 1000$ samples.}
  \label{fig:1}
\end{figure*}
\begin{figure*}
\centering
  \begin{subfigure}{}
    \includegraphics[width=0.48\linewidth]{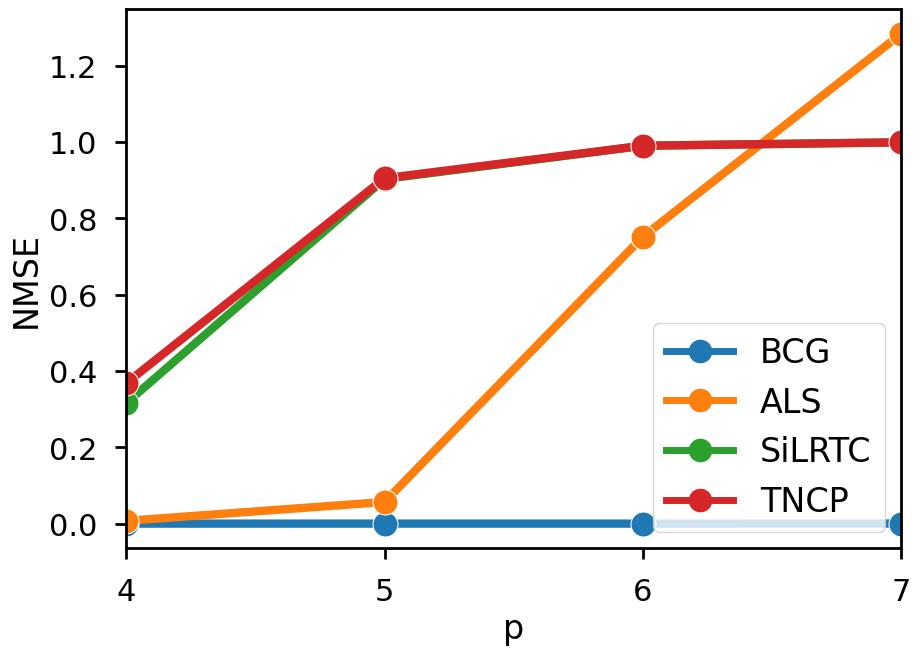}
 %   \label{fig:sub2a}
  \end{subfigure}
  \begin{subfigure}{}
    \includegraphics[width=0.48\linewidth]{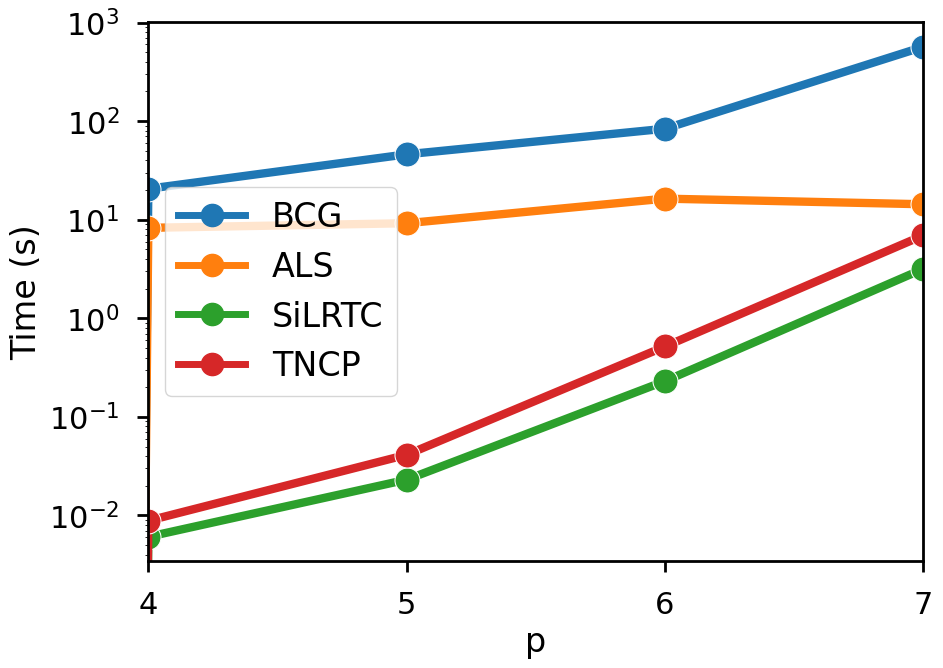}
%    \label{fig:sub2b}
  \end{subfigure}
  \caption{NMSE and computation time (in s) for increasing order tensors with size $10^{\times p}$ and $n=10,000$ samples. }
  \label{fig:2}

\label{fig:full1}
\end{figure*}

\begin{figure*}
\centering
  \begin{subfigure}{}
    \includegraphics[width=0.48\linewidth]{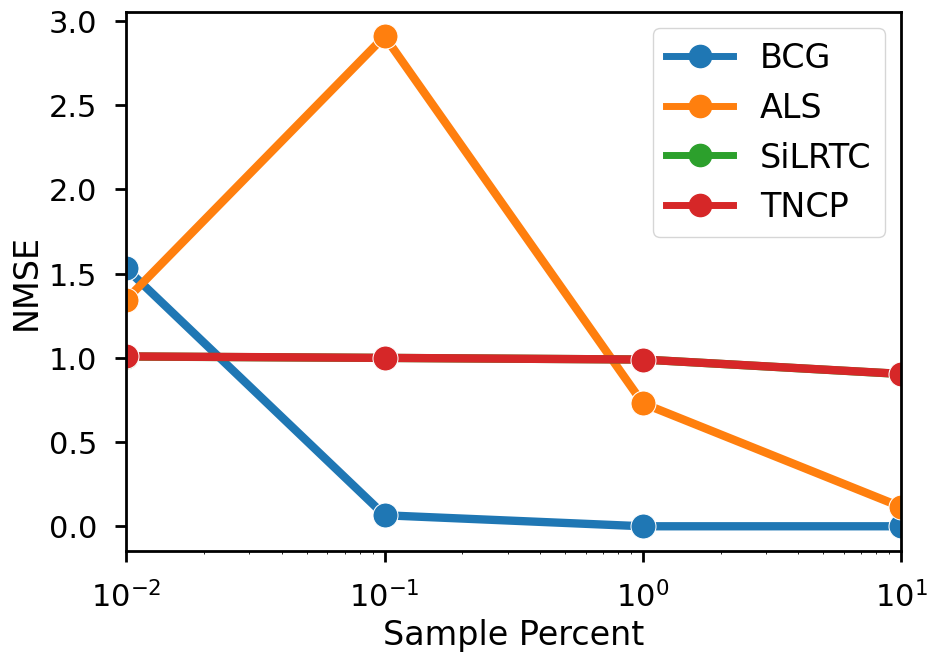}
%    \label{fig:sub3a}
  \end{subfigure}
  \begin{subfigure}{}
    \includegraphics[width=0.48\linewidth]{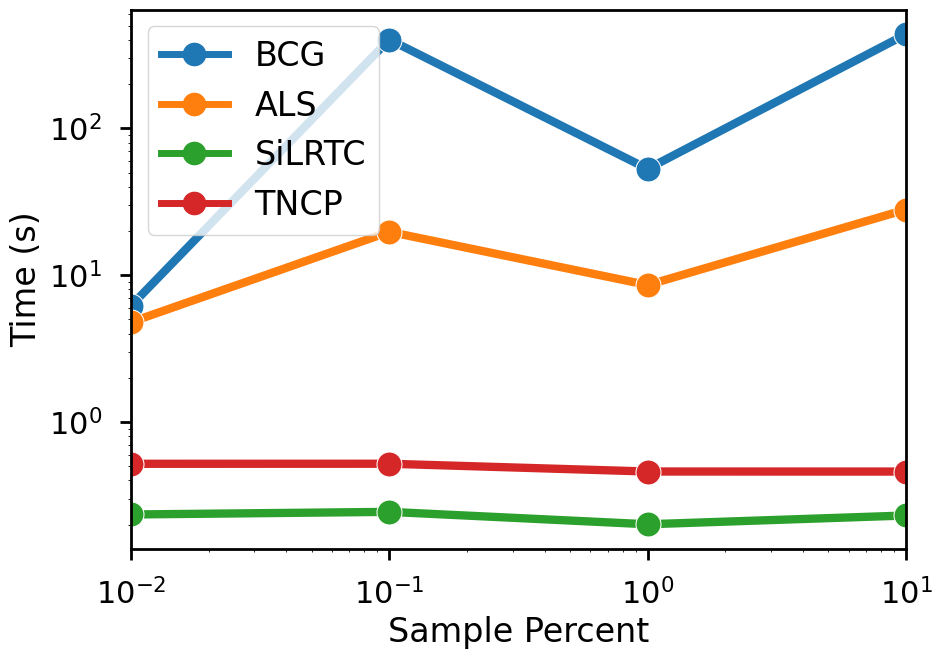}
 %   \label{fig:sub3b}
  \end{subfigure}
  \caption{NMSE and computation time (in s) for tensors with size $10^{\times 6}$ and increasing  $n$ samples.}
  \label{fig:3}
\end{figure*}
\begin{figure*}
\centering
  \begin{subfigure}{}
    \includegraphics[width=0.48\linewidth]{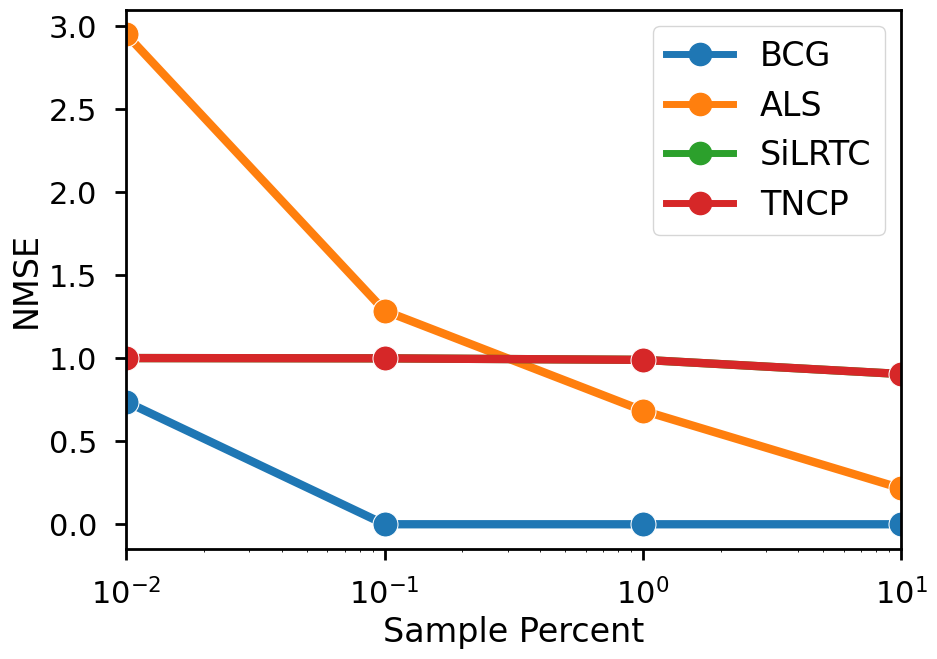}
%    \label{fig:sub4a}
  \end{subfigure}
  \begin{subfigure}{}
    \includegraphics[width=0.48\linewidth]{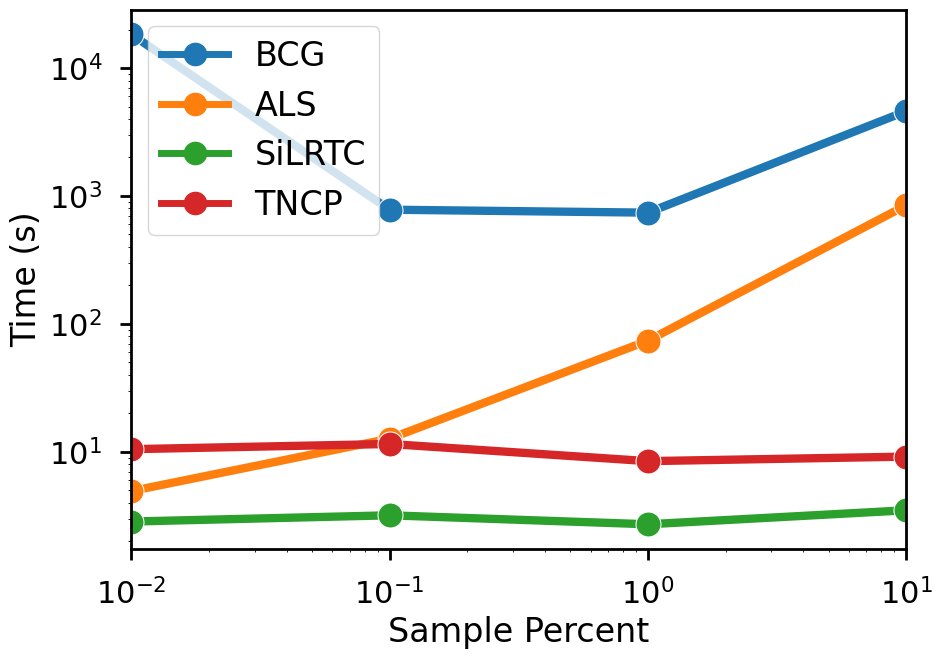}
 %   \label{fig:sub4b}
  \end{subfigure}
  \caption{NMSE and computation time (in s) for tensors with size $10^{\times 7}$ and increasing  $n$ samples.}
  \label{fig:4}

  \label{fig:full2}
\end{figure*}
\section{Numerical Experiments}\label{sec6}
In this section, we perform numerical experiments to demonstrate the efficacy and scalability of our tensor completion algorithm. The experiments were performed on computer server running a Linux operating system, with 16GB of RAM and an Intel Xeon Processor E5-2650L v3
(30M Cache, 1.80 GHz) that has 12 cores and became available in the year 2014. The algorithm was implemented in Python 3, and Gurobi v9.1 \cite{gurobi} was used as an integer optimization solver for the separation oracle  (\ref{eqn:intprog}). 

We also performed experiments on benchmark algorithms in tensor completion, including the often-called `workhorse' for numerical tensor problems, alternating least squares (ALS) \cite{kolda2009}, and two state-of-the-art algorithms implemented in the PyTen package \cite{song2019tensor}, known as the simple low-rank tensor completion (SiLRTC) algorithm \cite{liu2012tensor} and the trace norm regularized \textsc{cp} decomposition (TNCP) algorithm \cite{liu2014trace}. PyTen is available at \url{https://github.com/datamllab/pyten} under a GPL 2 license.

The true tensor $\psi$ is constructed in each experiment by taking a random convex combination of a random set of ten points from $\mathcal{S}_1$. This setup ensures that $\|\psi\|_\pm\leq 1$ and $\rank(\psi) \leq 10$, which are provided as ground truth values during the experiments. Each experiment was performed with 100 repetitions. We recorded the normalized mean squared error (NMSE) to quantify the accuracy of tensor completion by each algorithm. NMSE is given by $\|\widehat{\psi}-\psi\|_F^{\ 2}/\|\psi\|_F^{\ 2}$, which is a stricter measure than the error metric used in Corollary \ref{cor:mtrs} because the statistical guarantee is not normalized. 

Aiming to minimize the influence of hyper-parameter selection in the experiment results, we use the ground truth values during the numerical experiments when possible. In particular, constructing the true tensor $\psi$ following the procedure as described above is useful for providing the ground truth values of $\lambda$ in our algorithm and the value of $k$ in ALS and TNCP. We also note that ALS tends to perform better under L2 regularization \cite{navasca2008swamp}, and thus we selected the L2 regularization hyperparameter of ALS such that it was favorable to the accuracy of ALS.

\subsection{Increasing Tensor Dimension}
\label{sec:tot}
The first set of results in Figure \ref{fig:1} is on tensors of order $p=3$ with dimensions increasing from $r=10$ to $r=100$. In each experiment, we observe $n=1000$ samples, allowing for repetition in indices. Figure \ref{fig:1} shows that our approach yields greater accuracy in a lower size of $r$ while all algorithms do not perform at a satisfactory level (NMSE below 1) as $r$ increases close to 100. Although our algorithm takes more computation time, it converges on the order of seconds to minutes for all dimensions. We note that TNCP and SiLRTC's NMSE values converge to $1$ with increasing $r$ as a naive solution, whose entries are equal to the average of all $y\langle i \rangle$, will lead to an NMSE of 1.

\subsection{Increasing Tensor Order}
The second set of results in Figure \ref{fig:2} is on tensors of increasing order $p$ with dimension $r_{i}=10$ for $i= 1, \dots, p$. We observed $n=10,000$ samples, again with indices sampled at random with replacement. This set of results shows that our algorithm achieves consistently higher accuracy as compared to the other methods while requiring more computation times. Yet, even for tensors with $10^7$ entries, our algorithm is still able to converge within computation times on the order of minutes.

\subsection{Increasing Sample Size}
Our last set of results in Figures \ref{fig:3} and \ref{fig:4} is on tensors of sizes $10^{\times 6}$ and $10^{\times 7}$. In each experiment, the sample size is increased by one order of magnitude according to the values given in Figures \ref{fig:3} and \ref{fig:4} as the percentage of total entries, starting from $0.01\%$ (using random sampling with replacement). The results demonstrate that our algorithm achieves considerably higher accuracy while requiring greater computation time. Nevertheless, our algorithm is able to converge within minutes for most cases, except for the case where the sample percent is $10^{-1}\%$ for a tensor of size $10^7$ (i.e., approximately five hours).

\section{Conclusion} \label{sec7}
We define a new tensor norm using the gauge of a specific polytope to develop an algorithm for (general) tensor completion. The algorithm successfully resolves the tension between practical computation and the information-theoretic rate: our approach provably converges globally in a linear number of oracle calls to integer linear optimization while satisfying the information-theoretic sample complexity rate. Numerical experiments further demonstrated the efficacy and scalability of the algorithm. Next steps include efforts to further accelerate the algorithm, such that it can attain its high performance within computation times similar to benchmark algorithms.

%After designing an algorithm for general tensors, we do extensive simulations to evaluate its performance. We see that our algorithm also achieves the information-theoretic rate like in \cite{bugg2022nonnegative}. However, we also observe the increase in the complexity of the problem as our algorithm becomes slightly more computationally time-intensive. 

\section{Statements and Declarations}
The authors have no conflicts of interest to declare that are relevant to the content of this article. This material is based upon work supported by the National Science Foundation under Grant DGE-2125913 and Grant CMMI-1847666.
\bibliography{ref}

\begin{thebibliography}{10}
\providecommand{\url}[1]{\texttt{#1}}
\providecommand{\urlprefix}{URL }
\providecommand{\doi}[1]{https://doi.org/#1}

\bibitem{Acar_Dunlavy_Kolda_Mørup_2011}
Acar, E., Dunlavy, D.M., Kolda, T.G., Mørup, M.: Scalable tensor factorizations for incomplete data. Chemometrics and Intelligent Laboratory Systems  \textbf{106}(1),  41–56 (Mar 2011)

\bibitem{aswani2016}
Aswani, A.: Low-rank approximation and completion of positive tensors. SIAM Journal on Matrix Analysis and Applications  \textbf{37}(3),  1337--1364 (2016)

\bibitem{barak2016noisy}
Barak, B., Moitra, A.: Noisy tensor completion via the sum-of-squares hierarchy. In: Conference on Learning Theory. pp. 417--445. PMLR (2016)

\bibitem{bartlett2002}
Bartlett, P., Mendelson, S.: Rademacher and gaussian complexities: Risk bounds and structural results. J. Mach. Learn. Res.  (2002)

\bibitem{Bazerque_Mateos_Giannakis_2013}
Bazerque, J.A., Mateos, G., Giannakis, G.B.: Rank regularization and bayesian inference for tensor completion and extrapolation. IEEE Transactions on Signal Processing  \textbf{61}(22),  5689–5703 (Nov 2013)

\bibitem{braun2019blended}
Braun, G., Pokutta, S., Tu, D., Wright, S.: Blended conditonal gradients. In: International Conference on Machine Learning. pp. 735--743. PMLR (2019)

\bibitem{bugg2022nonnegative}
Bugg, C., Chen, C., Aswani, A.: Nonnegative tensor completion via integer optimization. Advances in Neural Information Processing Systems  \textbf{35},  10008--10020 (2022)

\bibitem{cai2019nonconvex}
Cai, C., Li, G., Poor, H.V., Chen, Y.: Nonconvex low-rank tensor completion from noisy data. Advances in neural information processing systems  \textbf{32} (2019)

\bibitem{chandrasekaran2012convex}
Chandrasekaran, V., Recht, B., Parrilo, P.A., Willsky, A.S.: The convex geometry of linear inverse problems. Foundations of Computational mathematics  \textbf{12}(6),  805--849 (2012)

\bibitem{Conforti_Cornuéjols_Zambelli_2014}
Conforti, M., Cornuéjols, G., Zambelli, G.: Integer Programming, Graduate Texts in Mathematics, vol.~271. Springer International Publishing, Cham (2014)

\bibitem{drenick1992multilinear}
Drenick, R.: Multilinear programming: Duality theories. Journal of optimization theory and applications  \textbf{72}(3),  459--486 (1992)

\bibitem{Duarte_Baraniuk_2012}
Duarte, M.F., Baraniuk, R.G.: Kronecker compressive sensing. IEEE Transactions on Image Processing  \textbf{21}(2),  494–504 (Feb 2012)

\bibitem{friedland2016computational}
Friedland, S., Lim, L.H.: The computational complexity of duality. SIAM Journal on Optimization  \textbf{26}(4),  2378--2393 (2016)

\bibitem{gandy2011tensor}
Gandy, S., Recht, B., Yamada, I.: Tensor completion and low-n-rank tensor recovery via convex optimization. Inverse problems  \textbf{27}(2),  025010 (2011)

\bibitem{Ge_Caverlee_Lu_2016}
Ge, H., Caverlee, J., Lu, H.: Taper: A contextual tensor-based approach for personalized expert recommendation. In: Proceedings of the 10th ACM Conference on Recommender Systems. p. 261–268. RecSys ’16, Association for Computing Machinery (Sep 2016)

\bibitem{gurobi}
{Gurobi Optimization, LLC}: {Gurobi Optimizer Reference Manual} (2021), \url{https://www.gurobi.com}

\bibitem{Hendrickx_Olshevsky_2010}
Hendrickx, J.M., Olshevsky, A.: Matrix p -norms are np-hard to approximate if $p\neq1,2,\infty$. SIAM Journal on Matrix Analysis and Applications  \textbf{31}(5),  2802–2812 (2010). \doi{10.1137/09076773X}

\bibitem{hillar2013}
Hillar, C., Lim, L.H.: Most tensor problems are np-hard. J. ACM  \textbf{60}(6),  45:1--45:39 (Nov 2013). \doi{10.1145/2512329}, \url{http://doi.acm.org/10.1145/2512329}

\bibitem{Jaggi_2013}
Jaggi, M.: Revisiting frank-wolfe: Projection-free sparse convex optimization. In: Proceedings of the 30th International Conference on Machine Learning. p. 427–435 (2013)

\bibitem{Jain_Gutierrez_Haupt_2017}
Jain, S., Gutierrez, A., Haupt, J.: Noisy tensor completion for tensors with a sparse canonical polyadic factor. In: 2017 IEEE International Symposium on Information Theory (ISIT). p. 2153–2157 (Jun 2017)

\bibitem{Javed_Bouwmans_Jung_2015}
Javed, S., Bouwmans, T., Jung, S.K.: Stochastic decomposition into low rank and sparse tensor for robust background subtraction. In: 6th International Conference on Imaging for Crime Prevention and Detection (ICDP-15). p. 1–6 (Jul 2015)

\bibitem{Karatzoglou_Amatriain_Baltrunas_Oliver_2010}
Karatzoglou, A., Amatriain, X., Baltrunas, L., Oliver, N.: Multiverse recommendation: n-dimensional tensor factorization for context-aware collaborative filtering. In: Proceedings of the fourth ACM conference on Recommender systems. p. 79–86. Association for Computing Machinery (Sep 2010)

\bibitem{kolda2009}
Kolda, T., Bader, B.: Tensor decompositions and applications. SIAM Review  \textbf{51}(3),  455--500 (2009)

\bibitem{kressner2014low}
Kressner, D., Steinlechner, M., Vandereycken, B.: Low-rank tensor completion by riemannian optimization. BIT Numerical Mathematics  \textbf{54}(2),  447--468 (2014)

\bibitem{lecue2013empirical}
Lecu{\'e}, G.: Empirical risk minimization is optimal for the convex aggregation problem. Bernoulli  \textbf{19}(5B),  2153--2166 (2013)

\bibitem{ledoux1991}
Ledoux, M., Talagrand, M.: Probability in Banach Spaces: Isoperimetry and Processes. Springer (1991)

\bibitem{lewis2017quadratic}
Lewis, M., Glover, F.: Quadratic unconstrained binary optimization problem preprocessing: Theory and empirical analysis. Networks  \textbf{70}(2),  79--97 (2017)

\bibitem{li2008image}
Li, X., Gunturk, B., Zhang, L.: Image demosaicing: A systematic survey. In: Visual Communications and Image Processing 2008. vol.~6822, p. 68221J. International Society for Optics and Photonics (2008)

\bibitem{liu2012tensor}
Liu, J., Musialski, P., Wonka, P., Ye, J.: Tensor completion for estimating missing values in visual data. IEEE transactions on pattern analysis and machine intelligence  \textbf{35}(1),  208--220 (2012)

\bibitem{liu2014trace}
Liu, Y., Shang, F., Jiao, L., Cheng, J., Cheng, H.: Trace norm regularized candecomp/parafac decomposition with missing data. IEEE transactions on cybernetics  \textbf{45}(11),  2437--2448 (2014)

\bibitem{massart2000}
Massart, P.: Some applications of concentration inequalities to statistics. Annales de la facult\'{e} des sciences de Toulouse S\'{e}r. 6  \textbf{9}(2),  245--303 (2000)

\bibitem{montanari2018spectral}
Montanari, A., Sun, N.: Spectral algorithms for tensor completion. Communications on Pure and Applied Mathematics  \textbf{71}(11),  2381--2425 (2018)

\bibitem{navasca2008swamp}
Navasca, C., De~Lathauwer, L., Kindermann, S.: Swamp reducing technique for tensor decomposition. In: 2008 16th European Signal Processing Conference. pp.~1--5. IEEE (2008)

\bibitem{nemirovski2000topics}
Nemirovski, A.: Topics in non-parametric statistics. Ecole d’Et{\'e} de Probabilit{\'e}s de Saint-Flour  \textbf{28}, ~85 (2000)

\bibitem{pan2023}
Pan, W., Aswani, A., Chen, C.: Accelerated non-negative tensor completion via integer programming. Frontiers in Applied Mathematics and Statistics  \textbf{9} (2023). \doi{10.3389/fams.2023.1153184}, \url{https://www.frontiersin.org/articles/10.3389/fams.2023.1153184}

\bibitem{qi2016}
Qi, Y., Comon, P., Lim, L.H.: Uniqueness of nonnegative tensor approximations. IEEE Transactions on Information Theory  \textbf{62}(4),  2170--2183 (2016). \doi{10.1109/TIT.2016.2532906}

\bibitem{rao2015}
Rao, N., Shah, P., Wright, S.: Forward–backward greedy algorithms for atomic norm regularization. IEEE Transactions on Signal Processing  \textbf{63}(21),  5798--5811 (2015). \doi{10.1109/TSP.2015.2461515}

\bibitem{rockafellar2009}
Rockafellar, R., Wets, R.: Variational Analysis. Springer (2009)

\bibitem{Rockafellar_2015}
Rockafellar, R.T.: Convex Analysis: (PMS-28). Princeton University Press (Apr 2015)

\bibitem{Rohn_2000}
Rohn, J.: Computing the norm $\|a\|_{\infty,1}$ is np-hard. Linear and Multilinear Algebra  \textbf{47}(3),  195–204 (2000). \doi{10.1080/03081080008818644}

\bibitem{Plas_De_Moor_Suykens_2011}
Signoretto, M., Van~de Plas, R., De~Moor, B., Suykens, J.A.K.: Tensor versus matrix completion: A comparison with application to spectral data. IEEE Signal Processing Letters  \textbf{18}(7),  403–406 (Jul 2011)

\bibitem{desilva2008}
de~Silva, V., Lim, L.: Tensor rank and the ill-posedness of the best low-rank approximation problem. SIAM J. Matrix Anal. Appl.  \textbf{30}(3),  1084--1127 (2008). \doi{10.1137/06066518X}, \url{http://dx.doi.org/10.1137/06066518X}

\bibitem{song2019tensor}
Song, Q., Ge, H., Caverlee, J., Hu, X.: Tensor completion algorithms in big data analytics. ACM Transactions on Knowledge Discovery from Data (TKDD)  \textbf{13}(1),  1--48 (2019)

\bibitem{srebro2010}
Srebro, N., Sridharan, K., Tewari, A.: Smoothness, low noise and fast rates. In: Advances in Neural Information Processing Systems. pp. 2199--2207 (2010)

\bibitem{Tan_Cheng_Feng_Feng_Wang_Zhang_2013}
Tan, H., Cheng, B., Feng, J., Feng, G., Wang, W., Zhang, Y.J.: Low-n-rank tensor recovery based on multi-linear augmented lagrange multiplier method. Neurocomputing  \textbf{119},  144–152 (Nov 2013)

\bibitem{Tomasi_Bro_2005}
Tomasi, G., Bro, R.: Parafac and missing values. Chemometrics and Intelligent Laboratory Systems  \textbf{75}(2),  163–180 (Feb 2005)

\bibitem{tsybakov2003optimal}
Tsybakov, A.B.: Optimal rates of aggregation. In: Learning theory and kernel machines, pp. 303--313. Springer (2003)

\bibitem{yuan2016tensor}
Yuan, M., Zhang, C.H.: On tensor completion via nuclear norm minimization. Foundations of Computational Mathematics  \textbf{16}(4),  1031--1068 (2016)

\bibitem{Zafarani_Abbasi_Liu_2014}
Zafarani, R., Abbasi, M.A., Liu, H.: Social Media Mining: An Introduction. Cambridge University Press, 1 edn. (Apr 2014)

\bibitem{zhang2019robust}
Zhang, X., Wang, D., Zhou, Z., Ma, Y.: Robust low-rank tensor recovery with rectification and alignment. IEEE Transactions on Pattern Analysis and Machine Intelligence  \textbf{43}(1),  238--255 (2019)

\end{thebibliography}
\bibliographystyle{splncs04}

%%%%%%%%%%%%%%%%%%%%%%%%%%%%%%%%%%%%%%%%%%%%%%%%%%%%%%%%%%%%%%%%%%%%%%%%%%%%%%%
%%%%%%%%%%%%%%%%%%%%%%%%%%%%%%%%%%%%%%%%%%%%%%%%%%%%%%%%%%%%%%%%%%%%%%%%%%%%%%%
% APPENDIX
%%%%%%%%%%%%%%%%%%%%%%%%%%%%%%%%%%%%%%%%%%%%%%%%%%%%%%%%%%%%%%%%%%%%%%%%%%%%%%%
%%%%%%%%%%%%%%%%%%%%%%%%%%%%%%%%%%%%%%%%%%%%%%%%%%%%%%%%%%%%%%%%%%%%%%%%%%%%%%%
%%%%%%%%%%%%%%%%%%%%%%%%%%%%%%%%%%%%%%%%%%%%%%%%%%%%%%%%%%%%%%%%%%%%%%%%%%%%%%%
%%%%%%%%%%%%%%%%%%%%%%%%%%%%%%%%%%%%%%%%%%%%%%%%%%%%%%%%%%%%%%%%%%%%%%%%%%%%%%%

\end{document}